\documentclass[11pt]{article}
\usepackage{amsmath,amsthm,amssymb,color,graphicx,url}
\newcommand{\remove}[1]{}
 \numberwithin{equation}{subsection}

\title{Teaching dimension, VC dimension and critical sets for Latin squares}

\author{
Hamed Hatami
\thanks{Supported by an NSERC grant.}
\\
School of Computer Science\\
McGill University, Montreal\\
\texttt{hatami@cs.mcgill.ca}
\and
Yingjie Qian
\\
The Department of Mathematics and Statistics\\
McGill University, Montreal\\
\texttt{yingjie.qian@mail.mcgill.ca}
}

\newtheorem{theorem}{Theorem}[section]
\newtheorem{lemma}[theorem]{Lemma}

\newtheorem{conjecture}[theorem]{Conjecture}
\newtheorem{corollary}[theorem]{Corollary}
\newtheorem{remark}[theorem]{Remark}

\def\TD{{\mathrm{TD}}}
\def\RTD{{\mathrm{RTD}}}
\def\VC{{\mathrm{VC}}}

\newcommand{\C}{\mathcal{C}}
\newcommand{\cL}{\mathcal{L}}
\newcommand{\T}{\mathcal{T}}

\title{Teaching dimension, VC dimension and critical sets for Latin squares}

\author{
Hamed Hatami
\thanks{Supported by an NSERC grant.}
\\
School of Computer Science\\
McGill University, Montreal\\
\texttt{hatami@cs.mcgill.ca}
\and
Yingjie Qian
\\
The Department of Mathematics and Statistics\\
McGill University, Montreal\\
\texttt{yingjie.qian@mail.mcgill.ca}
}
 
\begin{document}

\maketitle
\begin{abstract}
A critical set in an $n \times n$ Latin square is a minimal set of entries that uniquely identifies it among all Latin squares of the same size.  It is conjectured by Nelder in 1979, and later independently by Mahmoodian, and Bate and van Rees that the size of the smallest critical set is $\lfloor n^2/4\rfloor$. We prove a lower-bound of $n^2/10^4$  for sufficiently large $n$, and thus confirm the quadratic order predicted by the conjecture. {This improves a recent lower-bound of $\Omega(n^{3/2})$ due to Cavenagh and Ramadurai.}

From the point of view of computational learning theory, the size of the smallest critical set corresponds to the minimum teaching dimension of the set of Latin squares. We study two related notions of dimension from learning theory. We prove a lower-bound of $n^2-(e+o(1))n^{5/3}$ for both of the VC-dimension and the recursive teaching dimension.

\end{abstract}



\section{Introduction}

\paragraph*{Latin squares and critical sets:}  Recall that a \emph{Latin square} of order $n$ is an $n\times n$ array filled with elements from the set $\{1,2,\dots,n\}$  such that every element  occurs exactly once in each row and each column. Note that a Latin square $L$  can also be represented as the set of the triples
\begin{equation}
\label{eq:LatinTriples}
\{(i,j,k) \ | \ \mbox{the $(i,j)$-th entry is equal to $k$} \}.
\end{equation}

Following the notation of computational learning theory, we call a set of entries in a Latin square $L$ that uniquely identifies it among all Latin squares of order $n$, a \emph{teaching set} for $L$. The minimal teaching sets in Latin squares were introduced and studied under the name \emph{critical set} by statistician John Nelder~\cite{Nel}, and  they have been studied extensively since then. We refer the reader to the two surveys~\cite{MR1393712}~and~\cite{MR2453264} for more on this topic.

Note that a partially filled Latin square can determine the values of certain empty entries $(i,j)$ in a straightforward manner: if all the values  $\{1,\ldots,n\} \setminus \{k\}$ already appear in the $i$-th row and $j$-th column, then the $(i,j)$-th entry is determined to be $k$. One can start from a partially filled Latin square $P$ and iteratively set the values of the entries that are determined in this manner. If this finally leads to a full Latin square $L$, then  $P$ is called a \emph{strong teaching set} for $L$. Obviously every strong teaching set is also a teaching set. Bate and van Rees~\cite{MR1724489} showed that every strong teaching set is of size at least $\lfloor n^2/4\rfloor$. Figure~\ref{fig:ex} illustrates an example of a strong teaching set of size $\lfloor n^2/4\rfloor$ for a $4 \times 4$ Latin square.

\begin{figure}[h]
\vspace{0.5cm}
\caption[]{An example of a strong teaching set for a $4 \times 4$ Latin square.}
\label{fig:ex}
\begin{center}
\begin{tabular}{ | c | c | c | c |}
\hline
  1& & &  \\ \hline
   &2& &  \\ \hline
   & & &  \\ \hline
   &4&2& \\ \hline
\end{tabular}
\quad
\begin{tabular}{ | c | c | c | c |}
\hline
  1&3&4&2 \\ \hline
  4&2&1&3 \\ \hline
  2&1&3&4 \\ \hline
  3&4&2&1 \\ \hline
\end{tabular}
\end{center}
\end{figure}

Moreover, Bate and van Rees~\cite{MR1724489} conjectured that this bound holds for every teaching set. This was also independently conjectured  earlier by Nelder\footnote{John Nelder: Private communication to Jennifer Seberry (1979).}, and Mahmoodian~\cite{MR1492638}.

\begin{conjecture} \label{conj:critical}
Every critical set for a Latin square of order $n$ is of size at least $\lfloor n^2/4\rfloor$.
\end{conjecture}

The existence of  critical sets of size $\lfloor n^2/4\rfloor$ was shown by Curran and van Rees~\cite{MR541920} and Cooper, Donovan and Seberry~\cite{MR1129273}. However, despite several efforts, there has been little progress towards resolving this conjecture. Fu, Fu, and Rodger~\cite{MR1468170} showed a lower-bound of $\lfloor (7n-3)/6 \rfloor$ for $n\geq20$. This bound was
improved by Horak, Aldred, and Fleischner~\cite{MR1468170} to $\lfloor (4n-8)/3\rfloor$ for $n\geq8$. {Cavenagh~\cite{MR2330083} gave the first superlinear lower-bound of $n \lfloor (\log n)^{1/3}/2\rfloor$ in 2007. Recently, Cavenagh and Ramadurai~\cite{CR} improved this bound to $\Omega(n^{3/2})$.}


In Theorem~\ref{thm:scs} below, we use recent results of Barber, {K{\"u}hn}, Lo, Osthus and Taylor~\cite{BKLOT} and Dukes~\cite{Duk} about edge-decomposition of graphs into triangles to show that for sufficiently large $n$, every critical set in a Latin square is of size at least $n^2/10^4$, thus establishing that as it was predicted in Conjecture~\ref{conj:critical}, the size of the smallest critical set is of quadratic order.

\begin{theorem}
\label{thm:scs}
For sufficiently large $n$, every critical set for a Latin square of order $n$ is of size at least $10^{-4}n^2$.
\end{theorem}

\paragraph*{VC, teaching, and recursive teaching dimensions:}  The notion of teaching set for Latin squares, as defined above, is quite natural, and can be easily defined for other combinatorial objects. Indeed similar notions have been defined and studied independently under various names in different contexts. For example, the term \emph{defining set} is used for block designs and graph colorings,  and the term  \emph{forcing set}, coined by Harary~\cite{MR1223833}, is used for other concepts such as perfect matchings, dominating sets, and geodetics {(see} the survey~\cite{MR2011736}).

The general concept of identifying an object by a small set of its attributes arises naturally  in the area of computational learning theory. Consider a finite set $\Omega$, and let $\mathcal{F}(\Omega)$ denote the power set of $\Omega$. In computational learning theory, a subset $\C \subseteq \mathcal{F}(\Omega)$ is refered to as a \emph{concept class}, and the elements $c \in {\C}$ are called \emph{concepts}. A set $S \subseteq \Omega$ is called a \emph{teaching set} for a concept $c \in \C$ if $c \cap S$ uniquely identifies $c$ among all other concepts. In other words, $(c \cap S) \neq (c' \cap S)$ for every concept $c' \neq c$. The notion of a teaching set  was independently introduced by Goldman and Kearns~\cite{MR1322630}, Shinohara and Miyano~\cite{SM} and Anthony et al.~\cite{ABCS}. It has also been studied under the names \emph{witness set} by Kushilevitz et al. in~\cite{MR1370141}, \emph{discriminant} in~\cite{Nat}, and \emph{specifying set} in~\cite{ABCS}.

Recall from (\ref{eq:LatinTriples}) that every Latin square of order $n$ can be represented as a subset of $\{1,\ldots,n\}^3$. Hence the set $\cL_n$ of all Latin squares of order $n$ can be considered as a concept class. It is worth noting that our definition of a teaching set for a Latin square coincides with its definition when $\cL_n$ is considered as a concept class.

The concept of a teaching set naturally gives rise  to various notions of dimension associated to concept classes. Let $\TD(c;\C)$ denote the smallest size of a teaching set for a concept $c \in \C$. The \emph{teaching dimension} and the \emph{minimum teaching dimension} of a concept class $\C$ are respectively defined as $\TD(\C)=\max\limits_{c\in\C}\TD(c;\C)$ and  $\TD_{\min}(\C)=\min\limits_{c\in\C}\TD(c;\C)$. It turns out that for some purposes, due to its local nature, the  minimum teaching dimension do not capture the characteristics of teaching and learning, and thus  the related notion of \emph{recursive teaching dimension} is often considered:
$$\RTD(\C)=\max\limits_{\C'\subseteq\C}\TD_{\min}(\C').$$
Note that $\TD_{\min}(\C) \le \RTD(\C)\leq \TD(\C)$ for every concept class $\C$.

Finally let us recall one of the most celebrated notions of dimension associated to a concept class, i.e.  its VC dimension (for Vapnik-Chervonenkis dimension). A subset $S \subseteq \Omega$ is said to be \emph{shattered} by $\C$ if for every $T \subseteq S$ there exists a concept $c$ with $c \cap S=T$. The size of the largest set shattered by $\C$ is called the \emph{VC-dimension} of $\C$.  Recently in~\cite{CCT}, using a surprisingly short argument, Chen, Cheng and Tang showed that  $\RTD(\C) \leq2^{d+1}(d-2)+d+4$, where $d=\VC(\C)$.

\paragraph*{VC, teaching, and recursive teaching dimensions for Latin Squares:}
Our main result, Theorem~\ref{thm:scs}, says that  $\TD_{\min}(\cL_n) \ge 10^{-4}n^2$ for sufficiently large $n$.  Ghandehari, Hatami and Mahmoodian~\cite{MR2136056} showed that every Latin square contains a critical set of size at most $n^2-\frac{\sqrt\pi}{2}n^{3/2}$, and moreover there are Latin squares with no critical sets of size smaller than $n^2-(e+o(1))n^{5/3}$. In other words, for sufficiently large $n$, we have
$$n^2-(e+o(1))n^{5/3}\leq \TD(\cL_n)\leq n^2-\frac{\sqrt\pi}{2}n^{3/2}.$$

On the other hand, $\RTD(\cL_n)$ does not seem to correspond to any of the previously studied parameters related to critical sets. In Theorem~\ref{thm:RT} below, we show that one can adopt the argument of~\cite{MR2136056} to obtain a stronger result that $\RTD(\cL_n)\geq n^2-(e+o(1))n^{5/3}$. Surprisingly, a similar argument combined with a lemma of Pajor  (Lemma~\ref{lem:shatter}) implies the same bound for the VC-dimension.

\begin{theorem}
\label{thm:VCdimension}
The VC-dimension of the class of Latin squares of order $n$ is at least $n^2-(e+o(1))n^{5/3}$.
\end{theorem}

\begin{theorem}
\label{thm:RT}
The recursive teaching dimension of the class of Latin squares of order $n$ is at least $n^2-(e+o(1))n^{5/3}$.
 \end{theorem}

\section{Proof of Main Theorems}
In this section we present the proofs of our results, Theorem~\ref{thm:scs}, Theorem~\ref{thm:VCdimension}, and Theorem~\ref{thm:RT}.

\subsection{The size of the smallest critical set, Theorem~\ref{thm:scs}}
We give some remarks before proceeding to the proof of Theorem~\ref{thm:scs}. A graph $G$ has a \emph{$K_3$-decomposition} if its edge set can be partitioned into (edge-disjoint) copies of $K_3$. We call a $3$-partite graph $G$ \emph{balanced} if each part has the same number of vertices, and we call it \emph{locally balanced} if every vertex of $G$ has the same number of neighbours in each of the other two  parts (however, these numbers might be different for different vertices). The following theorem is immediate from results of  Barber, {K{\"u}hn}, Lo, Osthus and Taylor~\cite{BKLOT} and Dukes~\cite{Duk}.

\begin{theorem}[{See \cite[Corollary 1.6]{BKLOT} and \cite[Theorem 1.3]{Duk}}]
\label{thm:decomp}
Let $\gamma>0$ and $n>n_0(\gamma)$. Every balanced and locally balanced $3$-partite graph on $3n$ vertices with minimum degree at least $(101/52+\gamma)n$, admits a $K_3$-decomposition.
\end{theorem}

Noting that a Latin square of order $n$ is a $K_3$-decomposition of the complete $3$-partite graph $K_{n,n,n}$,  Barber, {K{\"u}hn}, Lo, Osthus and Taylor~\cite{BKLOT} obtained the following corollary to Theorem~\ref{thm:decomp}.

\begin{corollary}[{\cite{BKLOT}}]
Let $P$ be a partial Latin square of order $n\geq n_0$ such that every row, column, and symbol is used at most $0.0288n$ times. Then $P$ can be completed to a Latin square.
\end{corollary}

We will take a similar approach to prove Theorem~\ref{thm:scs}.

\begin{proof}[Proof of Theorem~\ref{thm:scs}]
Set $\epsilon=10^{-4}$. A \emph{partial Latin square} $P$ of order $n$ is a partially filled $n\times n$ array with elements chosen from $\{1,\ldots,n\}$ such that each element occurs at most once in every column and at most once in every row. In other words, some of the cells of the array are empty and the filled entries agree with the Latin property. The size of $P$, denoted by $|P|$, is the number of filled entries. {We need to show that  providing $n$ is sufficiently large, if a partial Latin square $P$ of size at most $\epsilon n^2$ can be completed to a Latin square $L$, then $P$ can also be completed to a different Latin square $L'$.}


For such a $P$, let $R,C,S$ be respectively the set of all rows, columns and symbols  in $P$ that have at least $\delta n$ filled entries, where $\delta=0.012$. We extend $P$ to a larger partial Latin square $P_1$ by completing all those rows, columns and symbols by filling the empty cells with the entries of $L$. Let $m=\max\{|R|,|C|,|S|\}$, and note $m \le \frac{\epsilon}{\delta}n \le 0.0084n$. We obtain $P_2$  by filling  $m-|R|$ additional rows, $m-|C|$ additional columns, and $m-|S|$ additional symbols with entries of $L$. Since $m+\delta n<n$,  exactly $m$ rows, $m$ columns, and $m$ symbols are all fully filled in $P_2$.


Let $(x,y,z) \in L \setminus P_2$.  Such an element exists because $|P_2|\leq |P|+3mn \le (\epsilon+\frac{3\epsilon}{\delta})n^2<n^2$. Let $z'$ be any symbol such that $(x,j,z'),(i,y,z')\not\in P_2$ for all $i,j\in \{1,\ldots,n\}$. Such a $z'$ exists because the number of symbols in the $x$-th row and the number of symbols in the $y$-th column of $P_2$ are  each at most $\delta n+2m$, and thus there are in total at most $2\delta n +4m<0.06n$ symbols appearing in the $x$-th row and the $y$-th column.

Let $P_3=P_2\cup\{(x,y,z')\}$  and we claim that $P_3$ can be completed to a Latin square. Note that $P_3$ still has exactly $m$ completed rows,  columns and symbols as filling $(x,y,z')$ in $P_2$ cannot create another complete row, column or symbol. Start from the complete $3$-partite graph $K_{n,n,n}$ where the vertices of each part are labeled with $\{1,\ldots,n\}$, and for every entry $(i,j,k) \in P_3$ remove the three edges of the triangle $(i,j,k)$ from the graph. Let $G$ be the resulting graph. Note that $G$ has $3m$ vertices of degree $0$ corresponding to the completed rows, columns and symbols in $P_3$. Ignoring the $0$-degree vertices, $G$ is  balanced and locally balanced, and it is of minimum degree at least $2n-2(\delta n+2m+1)>1.9426n>\frac{101}{52}(n-m)$. Hence by Theorem~\ref{thm:decomp}, it admits a $K_3$-decomposition, which in turn corresponds to a completion to a Latin square $L'$. Note that  $L'\neq L$ as the two Latin squares disagree on the $(x,y)$-th entry.
\end{proof}

{
\begin{remark}
A conjecture of Daykin and H\"aggkvist~\cite{MR777169} (see~\cite[Conjecture 1.3]{BKLOT}) suggests that  Theorem~\ref{thm:decomp} holds under the weaker condition that the   minimum degree of $G$ is at least $3n/2$. If this is true, the proof of Theorem~\ref{thm:scs} provides a better lower-bound of $2^{-7} n^2$ on the size of the smallest critical set. However, this is  still far from the conjectured bound of $\lfloor n^2/4 \rfloor$.  
\end{remark}}

\subsection{VC and recursive teaching dimension, Theorems~\ref{thm:VCdimension}~and~\ref{thm:RT}}
The van der Waerden conjecture, proved in~\cite{MR604006,MR638007,MR625097}, can be used to obtain a lower-bound for the number of Latin squares of order $n$.

\begin{lemma} [{\cite[Theorem 17.2]{MR1207813}}]
\label{lem:ln}
Let $\cL_n$ be the set of all Latin squares of order $n$. Then
$$|\cL_n|\geq\frac{(n!)^{2n}}{n^{n^2}}.$$
\end{lemma}

Ghandehari, Hatami and Mahmoodian~\cite[Theorem 3]{MR2136056} used Bregman's theorem~\cite{Bre} to obtain an upper-bound for the number of partial Latin squares of a given size.

\begin{lemma}[{\cite[Theorem 3]{MR2136056}}]
\label{lem:tnk}
 Let $\T_{n,k}$ be the set of all partial Latin squares of order $n$ and of size $k$. Then
$$|\T_{n,k}|\leq\binom{n^2}{k}\frac{n!^{2n-\frac{k}{n}}e^{n(3+\frac{\ln(2\pi n)^2}{4})}}{(n-\frac{k}{n})!^{2n}e^k}.$$
\end{lemma}

\paragraph*{The VC-dimension of Latin squares}

The most basic result concerning VC-dimension is the Sauer-Shelah lemma. This lemma that has been independently proved several times (e.g. in~\cite{MR0307902}), provides an upper-bound on the size of a concept class $\C \subseteq \mathcal{F}(\Omega)$ in terms of $|\Omega|$ and $\VC(\C)$. Formally it says $|\C| \le \sum_{i=0}^d {|\Omega| \choose i}$ where $d=\VC(\C)$. Note that for the set of $n \times n$ Latin squares $\cL_n \subseteq \{1,\ldots,n\}^3$, we have $|\Omega|=n^3$. Then it is not difficult to see that the Sauer-Shelah lemma together with Lemma~\ref{lem:ln} implies $\VC(\cL_n) \ge n^2 \left(\frac{1}{3}-o(1) \right)$. The $1/3$ factor in this bound is due to the cubic size of $|\Omega|$ in terms of $n$. To obtain the $n^2(1-o(1))$ bound of Theorem~\ref{thm:VCdimension}, we will use the following strengthening of the Sauer-Shelah lemma due to Pajor~\cite{MR903247}.

\begin{lemma}[~\cite{MR903247}]  \label{lem:shatter}
Every finite set family $\mathcal{F}$ shatters at least $|\mathcal{F}|$ sets.
\end{lemma}


\begin{proof}[Proof of Theorem~\ref{thm:VCdimension}]
We will prove that $n^2-e^{1+\frac{1}{\sqrt{n}}} n^{5/3}  < \VC(\cL_n)$ for sufficiently large $n$. Note that if a  set  $S \subseteq \{1,\ldots,n\}^3$ is shattered by $\cL_n$, then   in particular $S \cap L = S$ for some $L \in \cL_n$, and thus $S \subseteq L$. Hence every shattered set $S$ corresponds to a partial Latin square.  By Lemma~\ref{lem:shatter}, the set of all Latin squares of order $n$ shatters at least $|\cL_n|$ sets. It follows that for $d=\VC(\cL_n)$, we have
{\begin{equation}\label{eqn:shatterlatin}
\sum\limits_{k=0}^{d} |\T_{n,k}| \ge |\cL_n|.
\end{equation}}

Hence to prove $n^2-e^{1+\frac{1}{\sqrt{n}}} n^{5/3} < \VC(\cL_n)$, it suffices to show that  for every $k \le n^2-e^{1+\frac{1}{\sqrt{n}}} n^{5/3} $, we have $|\T_{n,k}|<\frac{|\cL_n|}{n^2}$,
or equivalently $|\cL_n| \le n^2 |\T_{n,k}|$ implies $k >n^2-e^{1+\frac{1}{\sqrt{n}}} n^{5/3}$.

We can follow a similar calculation  as in~\cite{MR2136056}: Assume $|\cL_n|  \le n^2 |\T_{n,k}|$. Then by Lemma~\ref{lem:ln} and Lemma~\ref{lem:tnk},  
\begin{equation} \label{eqn:less}
\frac{(n!)^{2n}}{n^{n^2}} \le n^2 \binom{n^2}{k}\frac{n!^{2n-\frac{k}{n}}e^{n(3+\frac{\ln(2\pi n)^2}{4})}}{(n-\frac{k}{n})!^{2n}e^k}.
\end{equation}
Setting $c=1-\frac{k}{n^2}$, and using  $\binom{n^2}{k}=\binom{n^2}{n^2-k}\leq\big(\frac{e}{c}\big)^{cn^2}$, we obtain
\begin{equation*}
\frac{n!^{n-cn}}{n^{n^2}} \le \frac{n^2e^{cn^2}e^{n\ln(2\pi n)^2}}{c^{cn^2}(cn)!^{2n}e^{n^2-cn^2}}.
\end{equation*}
Using {$n!\geq(\frac{n}{e})^n$}, we obtain
\begin{equation*}
\frac{n^{n^2-cn^2}}{e^{n^2-cn^2}n^{n^2}} \le \frac{n^2e^{3cn^2}e^{n\ln(2\pi n)^2}}{c^{cn^2}(cn)^{2cn^2}e^{n^2-cn^2}},
\end{equation*}
and thus
\begin{equation} \label{eqn:last1}
c^{3c}n^c  \le e^{3c}e^{\frac{\ln(2\pi n)^2}{n}}n^{\frac{2}{n^2}}.
\end{equation}

Fix a sufficiently large $n$. If $c=\frac{e^{1+\frac{1}{\sqrt{n}}}}{n^{1/3}}$, then $c^{3c}n^c>e^{3c}e^{\frac{\ln(2\pi n)^2}{n}}n^{\frac{2}{n^2}}$, and moreover $c^{3c}n^ce^{-3c}$ is an increasing function of $c$ in $[n^{-1/3},\infty)$. So  inequality~\eqref{eqn:last1} implies $c < \frac{e^{1+\frac{1}{\sqrt n}}}{n^{1/3}}$, which in turn shows $k > n^2 - e^{1+\frac{1}{\sqrt{n}}}n^{5/3}$ as desired. 

%
%
%
\end{proof}

\paragraph*{The recursive teaching dimension of Latin squares}
The proof of Theorem~\ref{thm:RT} will use  a similar counting argument as it was used in the proof of Theorem~\ref{thm:VCdimension}.

\begin{proof}[Proof of Theorem~\ref{thm:RT}]
Recall that $\cL_n$ denotes the set of all Latin squares of order $n$, and $\T_{n,k}$ denotes the set of all partial Latin squares of order $n$ and of size $k$.
Set $\cL = \cL_n$, and while  there are  partial Latin squares $P \in \T_{n,k}$ that have  unique extensions to full Latin squares $L \in \cL$,  remove such $L$'s from $\cL$. Repeat this process with the updated $\cL$ until no such partial Latin square  can  be found. Denote by $R$  the set of all Latin squares that are removed from the initial $\cL$, and note that  $|R| \le |\T_{n,k}|$. Note further that if $\cL \setminus R$ is not empty, then its minimum teaching dimension is at least $k$. We know from the proof of Theorem~\ref{thm:VCdimension} that $|\cL_n|> |\T_{n,k}|$ if $k \le n^2-(e+o(1))n^{5/3}$, and thus $\RTD(\cL_n)\geq \TD_{\min}(\cL_n \setminus R) \geq  n^2-(e+o(1))n^{5/3}$ as desired.
\end{proof}

\section{Concluding Remarks}
In Theorem~\ref{thm:scs} we proved that the size of the smallest critical set for Latin squares of order $n$ is of quadratic order, however  Conjecture~\ref{conj:critical} still remains unsolved.

In Theorems~\ref{thm:VCdimension}~and~\ref{thm:RT} we established a lower-bound of $n^2-(e+o(1))n^{5/3}$ for both VC-dimension and the recursive teaching dimension of the set of Latin squares of order $n$. One can easily obtain an upper-bound of the form $n^2-\Omega(n)$ for the VC-dimension, but obtaining a stronger upper-bound, and more ambitiously, determining the exact asymptotics of the VC-dimension seems highly nontrivial. For the teaching dimension and consequently recursive teaching dimension, a stronger upper-bound of $n^2-\frac{\sqrt{\pi}}{2}n^{3/2}$ follows from the results of~\cite{MR2136056}. Hence for sufficiently large $n$, 
$$n^2-(e+o(1))n^{5/3} \le \RTD(\cL_n)\leq \TD(\cL_n) \le n^2-\frac{\sqrt{\pi}}{2}n^{3/2}.$$
It would be interesting to improve either of the constants $5/3$ and $3/2$ appearing in the power of $n$ in the above bounds.

\section*{Acknowledgement}
We wish to thank Yaqiao Li for bringing our attention to the notions of teaching dimension and recursive teaching dimension. We would also like to thank John Bate, John Van Rees and Nicholas Cavenagh for their valuable comments and suggestions.

\bibliographystyle{alpha}
\bibliography{td_ls}

\end{document}